\documentclass[11pt]{amsart}
\usepackage{amsmath,amstext,amssymb,amsopn,amsthm}
\usepackage[T1]{fontenc} 
\usepackage[cp1250]{inputenc}
\usepackage{bbm}

\theoremstyle{plain}
\newtheorem{thm}{Theorem}[section]

\newtheorem{lemma}[thm]{Lemma}

\newtheorem{proposition}[thm]{Proposition}
\newtheorem{cor}[thm]{Corollary}

\theoremstyle{definition}

\theoremstyle{remark}
\newtheorem*{rem*}{Remark}

\newcommand{\R}{\mathbb{R}}
\newcommand{\Rp}{R^{\perp}}
\newcommand{\RR}{\mathbb{R}}
\newcommand{\N}{\mathbb{N}}

\renewcommand{\leq}{\leqslant}
\renewcommand{\geq}{\geqslant}
\newcommand{\norm}[1]{\Vert #1 \Vert}
\renewcommand{\k}{\mathbf{k}}
\newcommand{\ki}{\mathbf{k_1}}
\newcommand{\kii}{\mathbf{k_2}}

\renewcommand{\leq}{\leqslant}
\renewcommand{\geq}{\geqslant}

\def\({\left(}
\def\){\right)}
\def\[{\left[}
\def\]{\right]}
\def\<{\langle}
\def\>{\rangle}

\def \lm {\lambda}
\def \eps {\varepsilon}

\title{Uniform pointwise asymptotics  of solutions to quasi-geostrophic equation}
\author[T. Jakubowski]{Tomasz Jakubowski}
\address{Faculty of Pure and Applied Mathematics,
Wroc\l aw University of Science and Technology,
Wyb. Wyspia\'nskiego 27, 50-370 Wroc\l aw, Poland}
\email{tomasz.jakubowski@pwr.edu.pl}
\author[G. Serafin]{Grzegorz  Serafin}
\email{grzegorz.serafin@pwr.edu.pl}

\thanks{The paper is partially supported by the NCN grant 2015/18/E/ST1/00239}
\begin{document}

 \begin{abstract}
 We provide  two-sided pointwise estimates and uniform asymptotics  of the solutions to the subcritical quasi-geostrophic equation with initial data in $L^{2/(\alpha-1)}(\RR^2)$. Furthermore, we give upper bound of similar type for any derivative of the solutions. Initial data in 
  $L^{p}(\RR^2)$, $p>2/(\alpha-1)$, are also discussed.
 \end{abstract}

\maketitle

 \section{Introduction}

 In this paper we study the two-dimensional dissipative quasi-geostrophic equation 
\begin{align}\label{q-g}
\left\{
\begin{array}{l}
\theta_t+R^\perp\cdot\nabla\theta+(-\Delta)^{\alpha/2}\theta=0,\\
\theta(0,x)=\theta_0(x).
\end{array}\right.
\end{align}
Here, $R^\perp=(-R_2,R_1)$, where $R=(R_1,R_2)$ is the  two-dimensional Riesz transform given by 
$R_i=\frac{\partial}{\partial x_i}(-\Delta)^{-1/2}\theta$, $i\in\{1,2\}$. Throughout the paper we assume $\alpha\in(1,2)$ and $\theta$ is a mild solution to the initial value problem \eqref{q-g}, that is $\theta$ satisfies the following equation,
\begin{align}\label{eq:Duhamel}
\theta(t,\cdot)(x)&=P_t\theta_0(x)+\int_0^t\int_{\R^2}\nabla p(t-s,x-y)\cdot R^\perp\theta(s,y)\theta(s,y)dy\,ds.
\end{align}

For $\alpha=1$, the two-dimensional quasi-geostrophic equation is the analogue
of the 3D Navier-Stokes equation and solutions to both equations admit similar behaviour \cite{CW2}. The case $\alpha=1$ is therefore called critical exponent, while $\alpha\in(1,2)$ are subcritical exponents.

Solutions to the two-dimensional dissipative quasi-geostrophic equation  model several phenomena (see \cite{CMT, P}) and have been intensively studied for more than the last two decades. In 1995, Resnick \cite{R} proved existence of strong solutions for $\theta_0\in L^2(\RR^2)$ as well as the maximum principle
\begin{equation}\label{eq:1}
\|\theta(t,\cdot)\|_{L^p}\le \|\theta_0\|_{L^p}, 
\end{equation}
where $t\geq0$ and $1<p\le\infty$. This inequality has been improved in several directions  by deriving a precise decay rate of $\|\theta(t,\cdot)\|_{L^p}$, see e.g. \cite{CW, CC, CF, J,  NS1, NS2}. In \cite{CF} authors considered the initial condition $\theta_0\in L^p$ with $p\geq \frac2{\alpha-1}$ and obtained many interesting bounds  for $L^q$, $q\geq p$, norms of mild solutions to \eqref{q-g}. In particular, they showed that for $\theta_0\in L^{\frac2{\alpha-1}}$ and any  multi-index $\k\in\N^2$ 
\begin{align}
\lim_{t\rightarrow\infty} t^{\frac{|\k|}{\alpha}+\frac d\alpha\(\frac{\alpha-1}2-\frac1q\)}\|\nabla^\k\theta(t,\cdot)\|_q=0.
\end{align}
Under additional assumption $\theta_0\in L^1$, for every $\beta\in [0,\frac1\alpha)$ there is $C>0$ such that 
\begin{align*}
\|\nabla^\k\theta(t,\cdot)-\nabla^\k(P_t\theta_0)\|_q\le C\, t^{-\frac{|k|}{\alpha}-\frac d\alpha\(1-\frac1q\)-\beta},
\end{align*}
where $(P_t)_{t\geq0}$ is the semi-group generated by $\Delta^{\alpha/2}$.

Although all of the aforementioned results provide precise bounds for $L^p$ norms of the solutions, they do not say much about pointwise behaviour of these solutions. In particular, there are not known any 
lower bounds. In fact, this is rather common problem in the theory of nonlinear differential equations. Nevertheless, in this paper, we solve it for the dissipative quasi-geostrophic equation with nonnegative $\theta_0\in L^{\frac{2}{\alpha-1}}$ by giving two-sided pointwise estimates as well as some uniform asymptotics of mild  solutions. The main results of the paper are stated in the following theorems.
 \begin{thm}\label{thm:maincrit}Let $\theta_0\in L^{\frac{2}{\alpha-1}}(\RR^2)$ be nonnegative. There is a constant $C=C(\theta_0,\alpha)>1$ such that 
$$\frac1C P_t\theta_0(x)\le \theta(t,x)\le CP_t\theta_0(x).$$ 
\end{thm}
If we remove the nonnegativity condition, the upper bound $\theta(t,x)\le CP_t|\theta_0|$ holds (see Theorem \ref{thm:gradest}). Note that the semi-group $P_t$ and its kernel $p(t,x)$ are well known objects (see Section \ref{sec:stable} for the details).
\begin{thm}\label{thm:limt} For nonnegative $\theta_0\in L^{\frac2{\alpha-1}}(\RR^2)$, we have
\begin{align}\label{eq:limt}
	\lim_{t\rightarrow0}\left\|\frac{\theta(t,\cdot)}{P_t\theta_0}-1\right\|_{\infty}=\lim_{t\rightarrow\infty}\left\|\frac{\theta(t,\cdot)}{P_t\theta_0}-1\right\|_{\infty}=\lim_{|x|\rightarrow\infty}\sup_{t>0}\left|\frac{\theta(t,\cdot)}{P_t\theta_0}-1\right|=0.
\end{align}
\end{thm}
Finally, we complete these results  by establishing upper bounds for derivatives of the solutions: 
\begin{thm}\label{thm:gradest}For $\theta_0\in L^{\frac2{\alpha-1}}$ and any multi-index $\k\in\N\times\N$, there is $C=C(\theta_0,\k,\alpha)>0$ such that
\begin{align}\label{eq:gradest}
\left|\nabla^\k \theta(t,x)\right|\le  C t^{-|\k|/\alpha}P_t|\theta_0|(x).
\end{align}
\end{thm}
Note that $\nabla^\k  P_t\theta_0$ admits the same estimate (see \eqref{eq:estgradp}).
It turns out that the power $p=\frac2{\alpha-1}$ in the initial condition $\theta_0\in L^p$ is critical in some sense. One could  observe this phenomenon already in the paper \cite{CF}. Depending on $p$ is greater or less than $\frac2{\alpha-1}$, different difficulties occur and different behaviour of solutions is expected. Similar situation appears in the fractal Burgers equation, which has been studied by the authors in \cite{JS1, JS2} in case of (not only) critical power of the nonlinear drift term.  The methods developed there have been improved and adapted to the quasi-geostrophic equation. Nevertheless,  some ideas come from theory of linear perturbations of fractional Laplacian (see e.g. \cite{2007-KB-TJ-cmp, 2010-TJ-KS-jee}). In fact, the upper bound in \eqref{eq:1} is concluded from \cite{2013-YM-HM-jfa}, where also   linear equations have been considered.

The paper is organized as follows. In Section 2 we gather some properties of the heat kernel $p(t,x)$ for $\Delta^{\alpha/2}$ as well as some  basic facts and initial results for Riesz transform. Section 3 is devoted to estimates and asymptotics of solutions to (\ref{q-g}), while in Section 4 we prove the bound for theirs derivatives. 

Throughout the paper, we write $f\approx g$ ($f \lesssim g$ respectively) for $f,g\ge0$, if there
is a constant $c=c(\alpha,\theta_0)\ge1$ such that $c^{-1}f\le g\le c f$   ($f \le c g$ respectively) on their
common domain. The constants $c, C, c_i$, whose exact values are
unimportant, may be changed in each statement and proof. 
As usual, we write $a  \land b := \min(a,b)$ and $a \vee b := \max(a,b)$.
 \section{Preliminaries}

\subsection{Stable semigroup}
\label{sec:stable}
Throughout the paper we consider $\alpha \in (1,2)$. In this section we recall some well known results on the stable semigroup.  Let
\begin{equation}\label{e:levymea}
\nu(z)
=\frac{ \alpha 2^{\alpha-1}\Gamma\big(1 +\frac{\alpha}{2}\big)}{\pi \Gamma\big(1-\frac{\alpha}{2}\big)}|z|^{-2-\alpha}\,,\quad z\in \RR^2.
\end{equation}
For (smooth and compactly supported) test function $\varphi\in C^\infty_c(\RR^2)$, we define the fractional Laplacian  by
\begin{equation*}
  \Delta^{\alpha/2}\varphi(x):=-(-\Delta)^{\alpha/2}\varphi(x)=
  \lim_{\varepsilon \downarrow 0}\int_{\{|z|>\varepsilon\}}
  \left[\varphi(x+z)-\varphi(x)\right]\nu(z)\,dz\,,
  \quad
  x\in \RR^2\,.
\end{equation*}
In terms of the Fourier transform,
$\widehat{\Delta^{\alpha/2}\varphi}(\xi)=-|\xi|^{\alpha}\widehat{\varphi}(\xi)$.
Denote by $p(t,x-y)$ the fundamental solution of the equation $\partial_t u = \Delta^{\alpha/2}u$, that is, for fixed $y \in \R^2$, $u(t,x) = p(t,x-y)$ solves
$$\begin{cases}
\partial_t u = \Delta^{\alpha/2} u, &\qquad t>0, x \in \RR^2, \\
u(0,x) = \delta_y(x), &\qquad x \in \RR^2.	
\end{cases}
$$
It is well known that $p(\cdot,\cdot) \in C^\infty((0,\infty) \times {\RR^2})$ and $p(t,x)=p(t,-x)$ for any $t >0$ and $x\in \RR^2$. It also enjoys the following scaling and semi-group properties
\begin{align*}
p(t,x)&=t^{-2/\alpha}p(1,t^{-1/\alpha} x),\quad t>0,\; x\in \RR^2,\\
p(t,x)&=\int_{\RR^2}p(t-s, x-z) p(s,z) dz, \quad t>s>0,\; x\in \RR^2,
\end{align*}
%
and pointwise estimates,
\begin{equation}\label{eq:estp}
p(t,x) \approx \frac{t}{\(t^{1/\alpha}+|x|\)^{2+\alpha}} \approx t^{-2/\alpha} \land \frac{t}{|x|^{2+\alpha}},\quad t>0,\; x,y\in \RR^2.
\end{equation}
For any multindex $\k = (k_1,k_2) \in \mathbb{N}^2$ we denote
$$
 \nabla^\k f(x) = \frac{\partial^{|\k|}}{\partial x_1^{k_1}\partial x_2^{k_2}} f(x), \qquad x = (x_1,x_2) \in \RR^2,
$$
where $|\k| = k_1 + k_2$. By scaling property  and \cite[Lemma 3.1]{PSz-2010-mn} (see also \cite{KK-PS-2015-jmaa, TG-KS-preprint} for more general setting), 
\begin{equation}\label{eq:estgradp}
|\nabla^{\k} p(t,x)|\le c_\k t^{-\frac{|\k|}{\alpha}} p(t,x), \quad t>0,\; x,y\in \RR^2.
\end{equation}
From \eqref{eq:estgradp} we easily get the $L^p$-estimates:
\begin{equation}\label{eq:lpestp}
\|\nabla^{\k} p(t,\cdot)\|_p\lesssim t^{-\frac 2\alpha\(1-\frac1p\)-\frac{|\k|}{\alpha}}.
\end{equation}
We denote by $P_t$ the stable semigroup operator,
$$
P_t f(x) = \int_{\RR^2} p(t,x-y) f(y) dy, \qquad t>0,\, x \in \RR^2.
$$
The name 'stable' comes from an $\alpha$-stable process, which is  generated by  $\Delta^{\alpha/2}$  and the semigroup describes transition of probabilities
(see, e.g. \cite{MR2569321, 2010-KB-TG-MR-ap}). 
For $f\in L^{\frac2{\alpha-1}}(\RR^2)$ and  $p\in[\frac2{\alpha-1},\infty]$, the following estimate holds (\cite{W}),  
\begin{equation}\label{eq:PtfLp}
\|P_tf\|_p\lesssim t^{-\frac{(\alpha-1)}{\alpha} +\frac2{\alpha p}}\norm{f}_{\frac{2}{\alpha-1}}.
\end{equation}
In the lemma below, we note some additional decay properties.
%
\begin{lemma}
For  $f\in L^{\frac{2}{\alpha-1}}(\RR^2)$, we have 
\begin{align}
\lim_{t\rightarrow0}\|t^{\frac{\alpha-1}\alpha}P_tf\|_\infty&=0,  \label{eq:limPtf}\\
\lim_{t\rightarrow\infty}\|t^{\frac{\alpha-1}\alpha}P_tf\|_\infty &=0, \label{eq:limPtf2} \\
\lim_{|x|\rightarrow\infty}\sup_{t>0}\left|t^{\frac{\alpha-1}\alpha}P_tf(x)\right|&=0. \label{eq:limPtf3}
\end{align}
\end{lemma}
\begin{proof}
The limit \eqref{eq:limPtf} follows from \cite[(2.2)]{CF}. Next, for every $\eps>0$ there is $R>0$ such that $\norm{f\mathbbm{1}_{B(0,R)^c}}_\frac{2}{\alpha-1} <\eps$. By Young inequality and \eqref{eq:estgradp}, 
\begin{align*}
\norm{P_t(f \mathbbm{1}_{B(0,R)^c})}_\infty \le \norm{p(t,\cdot)}_\frac{2}{3-\alpha}\norm{f\mathbbm{1}_{B(0,R)^c}}_\frac{2}{\alpha-1} \le  c_1 t^{\frac{1-\alpha}\alpha}\eps.
\end{align*}
Hence,
\begin{align*}
	\|t^{\frac{\alpha-1}\alpha}P_tf\|_\infty 
	&\le t^{\frac{\alpha-1}\alpha} \left(\norm{P_t(f \mathbbm{1}_{B(0,R)})}_\infty + \norm{P_t(f \mathbbm{1}_{B(0,R)^c})}_\infty \right) \\
	&\le c_1 t^{\frac{\alpha-1}\alpha} \left(\norm{p(t,\cdot)_{\infty}}\norm{ \mathbbm{1}_{B(0,R)}}_\frac{2}{3-\alpha}\,\norm{f}_\frac{2}{\alpha-1} + t^{\frac{1-\alpha}\alpha}\eps \right) \\
	&\le c_2 (t^{\frac{\alpha-3}{\alpha}} + \eps),
\end{align*}
which yields \eqref{eq:limPtf2}. Finally, for $|x| > 2R$ and $|y|<R$, by \eqref{eq:estp}, we have $p(t,x-y) \lesssim t^{\frac{1-\alpha}{\alpha}}|x-y|^{\alpha-3} \lesssim t^{\frac{1-\alpha}{\alpha}}|x|^{\alpha-3}$. Therefore, for $|x|>2R$,
\begin{align*}
\sup_{t>0}\left|t^{\frac{\alpha-1}\alpha}P_tf(x)\right| &\le \sup_{t>0} t^{\frac{\alpha-1}\alpha}\left(\left|P_t(f\mathbbm{1}_{B(0,R)})(x)\right| + \norm{P_t(f\mathbbm{1}_{B(0,R)^c})}_\infty\right) \\
&\lesssim |x|^{\alpha-3} \norm{f\mathbbm{1}_{B(0,R)}}_1 + \eps \lesssim |x|^{\alpha-3} \norm{f}_\frac{2}{\alpha-1} + \eps
\end{align*}
and \eqref{eq:limPtf3} holds.
\end{proof}

\subsection{Riesz transform}
Let $R_i$, $i\in\{1,2\}$ be the Riesz transforms, i.e.
\begin{align*}
R_i f(x) = c\,P.V. \int_{\RR^2} \frac{y_i}{|y|^3} f(x-y) dy, \qquad y = (y_1,y 	_2) \in \RR^2,
\end{align*}
where $c$ is some constant and  $P.V.$ denotes the principal value of the integral. Let $R= (R_1,R_2)$ and $\Rp = (-R_2,R_1)$. It is clear that $|\Rp f| = |R f|$. 

It is well known that the Riesz transform is continuous on $L^p$ for $p\in(1,\infty)$, i.e. for $f\in L^p$ we have (see e.g. \cite[Corollary 4.8]{D})
\begin{equation}\label{eq:Rieszineq}
\|Rf\|_p\le c_p\|f\|_p.
\end{equation}

\begin{proposition}\label{prop:estRp}
For any multi-index $\k\in\N\times\N$ there is a constant $C>0$ such that  
\begin{equation}\label{eq:estRp}
|\Rp \nabla^{\k}p(t,x)|\le C t^{- \frac{|\k|}\alpha} \frac{1}{(t^{1/\alpha}+|x|)^2},  \qquad t>0, x \in \RR^2.
\end{equation}
\end{proposition}
\begin{proof}
It is easy to see that both sides of \eqref{eq:estRp} admit the scalling property $f(t,x)=t^{-(2+|\k|)/\alpha}f(1,t^{-1/\alpha}x)$.
 Hence, it is enough to consider $t=1$. First, let us write
\begin{align*}
|R_i\nabla^\k p(1,x)| & = c\left| P.V. \int_{\RR^2}\frac{y_i}{|y|^3}\nabla^\k p(1,x-y)dy\right|\\
&\le c\left|P.V. \int_{|y|\le1} \frac{y_i}{|y|^3}\nabla^\k p(1,x-y)dy\right| +c \int_{|y|>1} \frac{1}{|y|^2}\left|\nabla^\k p(1,x-y)\right|dy.
\end{align*}
It follows from (\ref{eq:estgradp}) that
\begin{align*}
\sup_{|w| \le 1}\left|\nabla(\nabla^\k p(1,x+w)\right| \lesssim  t^{-(|\k|+1)/\alpha} p(t,x).
\end{align*}
Hence, since
\begin{align*}
P.V. \int_{|y|\le1}\frac{y_i}{|y|^3}\nabla^\k p(1,x)dy = 0,
\end{align*}
the mean value theorem gives us
\begin{align*}
\left|P.V. \int_{|y|\le1}\frac{y_i}{|y|^3}\nabla^\k p(1,x-y)dy \right|&= \left|P.V. \int_{|y|\le1}\frac{y_i}{|y|^3}(\nabla^\k p(1,x-y)-\nabla^\k p(1,x))dy \right| \\
& = \left| \int_{|y|\le1}\frac{y_i}{|y|^3} y \cdot \nabla(\nabla^\k p(1,x+w_y) dy\right|  \\
& \le \int_{|y|\le1}\frac{|y|^2}{|y|^3}\sup_{|w| \le 1}\left|\nabla(\nabla^\k p(1,x+w)\right|dy  \lesssim p(1,x) \lesssim \frac{1}{x^2+1}. 
\end{align*}
Next,
\begin{align*}
\int_{|y|>|x|\vee 1 }\frac{1}{|y|^2}\left|\nabla^\k p(1,x-y) dy\right|&\lesssim \int_{|y|>1 \vee |x|}\frac{1}{1+|x|^2} p(1,x-y)dy \le \frac{1}{1+|x|^2},
\end{align*}
which gives \eqref{eq:estRp} for $|x| \le 1$.
Finally, for $1< |y| \le |x|$, we have $\frac{1+|y|^{2+\alpha}}{|y|^2} \le 2|y|^\alpha \le 2|x|^\alpha\le 2\frac{2+|x|^{2+\alpha}}{|x|^2}$, which yields $\frac{1}{|y|^2} \lesssim \frac{p(1,y)}{p(2,x)(1+|x|^2)}$. Thus,
\begin{align*}
\int_{1< |y| \le |x|}\frac{1}{|y|^2}\left|\nabla^\k p(1,x-y) dy\right|&\lesssim \int_{\RR^d}\frac{1}{1+|x|^2} \frac{p(1,y) p(1,x-y)}{p(2,x)}dy = \frac{1}{1+|x|^2}.
\end{align*}

\end{proof}

%

\begin{proposition}\label{prop:RPthetaest} For every $\k \in \N^2$ there is a constant $C_\k>0$ such that for all $t>0$,
\begin{equation}\label{eq:RPthetaLinfty}
\norm{\nabla^\k R^\perp P_t\varphi}_\infty \le C_\k t^{-\frac{|\k|+\alpha-1}{\alpha}} \norm{\varphi}_\frac{2}{\alpha-1}, \qquad \varphi\in L^{\frac{2}{\alpha-1}}(\RR^2).
\end{equation}
Furthermore
$$\lim_{t\rightarrow0}\|t^{\frac{\alpha-1}\alpha}R^\perp P_t\varphi\|_\infty=\lim_{t\rightarrow\infty}\|t^{\frac{\alpha-1}\alpha}R^\perp P_t\varphi\|_\infty=\lim_{|x|\rightarrow\infty}\sup_{t>0}\left|t^{\frac{\alpha-1}\alpha}R^\perp P_{t}\varphi(x)\right|=0.$$
\end{proposition}
\begin{proof}
First, by (\ref{eq:Rieszineq}) and (\ref{eq:lpestp}), we have
\begin{align*}
|\nabla^\k R_iP_t\varphi(x)|&=\left|\nabla^\k\int_{\R^2}p(t,x-y)R_i\varphi(y)dy\right|\\
&\le\int_{\R^2} |\nabla ^kp(t,x-y)R_i\varphi(y)| dy\\
&\le \|\nabla^\k p(t,\cdot)\|_{\frac2{3-\alpha}}\|R_i\varphi\|_{\frac2{\alpha-1}}\lesssim t^{-(|k|+\alpha-1)/\alpha}\|\varphi\|_{\frac2{\alpha-1}},
\end{align*}
which gives \eqref{eq:RPthetaLinfty}. Let us fix $\varepsilon>0$. 
There are $M_\varepsilon>0$ and $R_\varepsilon$ such that $\norm{\varphi \mathbbm{1}_{\{|\varphi|>M_\varepsilon\}}}_{\frac{2}{\alpha-1}} \le \eps$ and $\norm{\varphi \mathbbm{1}_{B(0,R_\eps)^c}}_{\frac{2}{\alpha-1}} \le \eps$.
Hence, by \eqref{eq:Rieszineq} and \eqref{eq:lpestp},
\begin{align}
\int_{|\varphi|> M_\varepsilon}|R_ip(t,x-y)\varphi(y)|dy &\lesssim 	\norm{p(t,\cdot)}_{\frac2{3-\alpha}}\(\int_{|\varphi|>M_\varepsilon}|\varphi(y)|^{\frac2{\alpha-1}}dy\)^{\frac{\alpha-1}2} \le \varepsilon\, t^{-\frac{\alpha-1}{\alpha}}, \label{eq1:RPthetaest}\\
\int_{|y|> R_\varepsilon}|R_ip(t,x-y)\varphi(y)|dy &\lesssim 	\norm{p(t,\cdot)}_{\frac2{3-\alpha}}\(\int_{|y|>R_\varepsilon}\varphi(y)^{\frac2{\alpha-1}}dy\)^{\frac{\alpha-1}2} \le \varepsilon\, t^{-\frac{\alpha-1}{\alpha}}. \label{eq2:RPthetaest}
\end{align}

Thus, by  \eqref{eq1:RPthetaest}, \eqref{eq:Rieszineq} and \eqref{eq:lpestp}, 
\begin{align*}
|R_iP_t\varphi(x)|&=\left|\int_{\R^2}R_ip(t,x-y)\varphi(y)dy\right|\\
&\le  \sqrt {M_\varepsilon}\int_{|\varphi|\le M_\varepsilon}|R_ip(t,x-y)|\sqrt{|\varphi(y)|}dy +\int_{|\varphi|> M_\varepsilon}|R_ip(t,x-y)\varphi(y)|dy \\
&\lesssim \sqrt {M_\varepsilon} \|R_ip(t,\cdot)\|_{\frac4{5-\alpha}}\|\varphi\|_{\frac2{\alpha-1}}+\varepsilon\, t^{-\frac{\alpha-1}{\alpha}} \\
&\lesssim \sqrt{M_\varepsilon} t^{-\frac{\alpha-1}{2\alpha}}+\varepsilon\, t^{-\frac{\alpha-1}{\alpha}},
\end{align*}
and consequently $\norm{t^{\frac{\alpha-1}\alpha}R^{\perp}P_t\varphi}_\infty \lesssim \sqrt{M_\varepsilon} t^{\frac{\alpha-1}{2\alpha}}+\varepsilon$, which proves the first limit from the assertion.

Next, by  (\ref{eq:lpestp}), (\ref{eq:estRp}), \eqref{eq1:RPthetaest} and \eqref{eq2:RPthetaest}, we get 
\begin{align*}
|R_iP_t\varphi|&=\left|\int_{\R^2}R_ip(t,x-y)\varphi(y)dy\right|\\
&\le   M_\varepsilon\int_{\substack{|\varphi|\le M_\varepsilon\\|y|\le R_\varepsilon}}|R_ip(t,x-y)|dy +\int_{|\varphi|> M_\varepsilon}|R_ip(t,x-y)\varphi(y)|dy +\int_{|y|> R_\varepsilon}|R_ip(t,x-y)\varphi(y)|dy \\
&\lesssim  {M_\varepsilon R_\varepsilon^2}\,{t^{-2/\alpha}}+\varepsilon\, t^{-\frac{\alpha-1}{\alpha}},
\end{align*}
which implies  $\lim_{t\rightarrow\infty}\|t^{\frac{\alpha-1}\alpha}R^{\perp}P_t\varphi\|_\infty=0$. 
By virtue of the previous two limits, it is enough to prove that for any $0<t_1<t_2<\infty$,
$$\lim_{|x|\rightarrow\infty}\sup_{t\in(t_1,t_2)}\left|R^{\perp}P_t\varphi(x)\right|=0.$$ 
By  (\ref{eq:Rieszineq}), (\ref{eq:estRp}), \eqref{eq2:RPthetaest} and H\"older inequality, we get for $|x|>R_\varepsilon$ and $t\in(t_1,t_2)$,
\begin{align*}
|R_iP_t\varphi(x)|&=\left|\int_{\R^2}R_ip(t,x-y)\varphi(y)dy\right|\\
&\lesssim   \int_{|y| > R_\varepsilon}|R_ip(t,x-y) \varphi(y)|dy +\int_{|y|\le R_\varepsilon}\frac{1}{(t_1^{1/\alpha}+|x-y|)^2}\varphi(y)dy \\
&\lesssim   \varepsilon\, t_1^{-\frac{\alpha-1}{\alpha}} +\frac{1}{(|x|-R_\varepsilon)^2}\(\int \mathbbm{1}_{\{|y| \le R_\varepsilon\}}dy\)^{\frac{3-\alpha}2}\|\varphi\|_{\frac2{\alpha-1}},
\end{align*}

which is arbitrary small for large $|x|$. This proves the last assertion.

\end{proof}

\section{Asymptotics and  estimates of solutions}

First, we recall some results from \cite{CF} concerning $L^p$ estimates of the solutions  to \eqref{q-g}. We assume below that $\theta_0\in L^{\frac2{\alpha-1}}$. For $p\in [\frac2{\alpha-1},\infty]$ we have (see  \cite[Prop. 3.2]{CF}) 
\begin{align}
t^{\frac{\alpha-1+|\k|}{\alpha} -\frac2{\alpha p}}\nabla^\k\theta\in BC((0,\infty), L^{p}(\R^2)), \label{eq:thetaLp}
\end{align}
where $\k\in\N\times\N$. In particular, for $p\in [\frac2{\alpha-1},\infty]$,
\begin{align}
\label{eq:thetaLpnorm} 
\norm{\theta(t,\cdot)}_p \lesssim  t^{-\frac{\alpha-1}{\alpha} + \frac2{\alpha p}}, \qquad t>0.	
\end{align}
Combining this with (\ref{eq:Rieszineq}), we get for $p\in[\frac2{\alpha-1},\infty)$, 
\begin{equation}\label{eq:lpRtheta}
\|R^\perp \theta(t,\cdot)\|_p\lesssim t^{-\frac{\alpha-1}{\alpha} + \frac2{\alpha p}}, \qquad t>0.
\end{equation}
We will need the following auxiliary lemma.
\begin{lemma}\label{lem:tmp1}
Let $p\ge\frac2{\alpha-1}$ and $q=\frac{p}{p-1}$. Assume that $f(s,\cdot)\in L^{q}$ and $g(s,\cdot) \in L^p$ satisfy
\begin{align*}
\norm{f(s,\cdot)}_q \le c_1 s^{-\frac{3}{\alpha} + \frac{2}{\alpha q}}, && \norm{g(s,\cdot)}_p \le c_2 s^{-\frac{\alpha-1}{\alpha} +\frac2{\alpha p}}.
\end{align*}
Then, there is a constant $C$ such that for $p \ge \frac{2}{\alpha-1}$,
\begin{align}\label{Eq1:tmp1}
	\int_{\R^2}  |f(t-s,x-y|\, |g(s,y)| dy \le C (t-s)^{-\(\frac1\alpha+\frac2{p\alpha}\)}s^{-\frac{\alpha-1}{\alpha} +\frac{2}{\alpha p}}\,, \qquad 0<s<t, \; x \in \RR^2.
\end{align}
Furthermore, for $t>0$, $x \in \RR^2$ and $p>\frac{2}{\alpha-1}$, we have  
\begin{align}\label{Eq2:tmp1}
	\int_0^t\int_{\R^2}  |f(t-s,x-y)|\, g	(s,y)|\,s^{-\frac{\alpha-1}\alpha}dy\,ds \le C\,  \mathcal{B}\(\tfrac{p(\alpha-1)-2}{p\alpha}, \tfrac{p(2-\alpha)+2+p}{\alpha p}\)t^{-\frac{\alpha-1}{\alpha}}.
\end{align}
\end{lemma}
\begin{proof} By H\"older inequality,
\begin{align*}
\int_{\R^2}|f(t-s,x-y)| |g(s,y)| dy &\le \|f(t-s,\cdot)\|_{\frac p{p-1}}\| g(s,\cdot)\|_{p}  \le c_1 c_2 (t-s)^{-\frac1\alpha-\frac2{\alpha p}}s^{-\frac{\alpha-1}{\alpha} + \frac{2}{\alpha p}}\,,
\end{align*}
which gives \eqref{Eq1:tmp1}. Furthermore, this implies
\begin{align*}
	\int_0^t\int_{\R^2}  |f(t-s,x-y)|\, g(s,y)|\,s^{-\frac{\alpha-1}\alpha}dy\,ds 
	&\le c \int_0^t (t-s)^{-\frac1\alpha - \frac2{p\alpha}}s^{-\frac{2\alpha-2}{\alpha} +\frac{2}{\alpha p}}\,ds\\
 &= c t^{-\frac{\alpha-1}{\alpha}}\int_0^1 (1-u)^{-\frac1\alpha-\frac2{p\alpha}} u^{-\frac{2\alpha-2}{\alpha} +\frac{2}{\alpha p}}\,du \\
 &= c \, \mathcal{B}\(\tfrac{p(\alpha-1)-2}{p\alpha}, \tfrac{(2-\alpha)p+2}{p\alpha}\) t^{-\frac{\alpha-1}{\alpha}}.
\end{align*}
\end{proof}

The next corollary is an immediate consequence of Lemma \ref{lem:tmp1}.

\begin{cor}\label{cor:tmp1}
Let $\theta$ be a solution to (\ref{q-g}) with $\theta_0\in L^{\frac{2}{\alpha-1}}$. For every $t>0$ and $p\in\big(\frac{2}{\alpha-1},\infty\big)$, we have 
\begin{align}
\int_0^t\int_{\R^2}  |\nabla p(t-s,\cdot)(x-y)|\, |R\theta(s,y)|\,s^{-\frac{\alpha-1}\alpha}dy\,ds &\le C t^{-\frac{\alpha-1}{\alpha}},  \label{Eq3:tmp1}\\
	\int_0^t\int_{\R^2}  |R\nabla p(t-s,\cdot)(x-y)|\, |R\theta(s,y)|\,s^{-\frac{\alpha-1}\alpha}dy\,ds &\le C t^{-\frac{\alpha-1}{\alpha}}. \label{Eq3:tmp2}
\end{align}

\end{cor}
\begin{proof}
Both of the bounds follow  from (\ref{eq:lpestp}), \eqref{eq:Rieszineq}, \eqref{eq:thetaLp} and \eqref{eq:lpRtheta} applied to \eqref{Eq2:tmp1}.
\end{proof}

The below-given bound extends (\ref{eq:lpRtheta}) to $p\in(1,\infty]$.
\begin{proposition}\label{prop:Rtheta-est} Assume $\theta_0\in L^{\frac2{\alpha-1}}$. There is a constant $C>0$ sucht that
\begin{equation}\label{eq:RthetaLinfty}
\|R^\perp\theta(t,\cdot)\|_\infty \le  C t^{-\frac{\alpha-1}{\alpha}}.
\end{equation}
\end{proposition}
\begin{proof}
By \eqref{eq:Duhamel}, we get
\begin{align*}
R_i\theta(t,\cdot)(x)&=R_iP_t\theta_0(x)+\int_0^t\int_{\R^2}R_i\nabla p(t-s,\cdot)(x-y)\cdot R^\perp\theta(s,y)\theta(s,y)dy\,ds.
\end{align*}
By Proposition \ref{prop:RPthetaest},  we have $\norm{R_iP_t\theta_0}_\infty \le ct^{-\frac{\alpha-1}{\alpha}}$ and the assertion follows from (\ref{eq:thetaLpnorm}) and \eqref{Eq3:tmp2}.
\end{proof}
Now, we will pass to the proof of pointwise upper bounds for solutions to \eqref{q-g}.	 Let  $L^{p,\lambda}(\RR^2)$ be the Morrey space, i.e.
\begin{align*}
L^{p,\lambda}(\RR^2) = \{f \in L^p(\RR^2) \colon \norm{f}_{L^{p,\lm}} := \sup_{r>0} \sup_{x \in \RR^2} r^{-\lambda} \int_{B(x,r) \cap \Omega} |f(z)|^p dz<\infty\}.
\end{align*}
The Morrey space is a Banach space with the norm $\norm{f}_{L^{p,\lm}}$.
For any Banach space $X$ we denote by $L^{p,\lambda}((0,\infty);X)$ the space of functions $f \colon (0,\infty) \to X$ such that
\begin{align*}
\norm{f}_{L^{p,\lm}((0,\infty);X)} := \sup_{0<s<t<\infty}  \left( (t-s)^{-\lambda}\int_s^t \norm{f(r)}_X^p dr\right)^{\frac1p} < \infty.
\end{align*} 
 It is also a Banach space with the norm $\norm{f}_{L^{p,\lm}((0,\infty);X)}$.
\begin{lemma}\label{lem:uppercrit}
Let $\theta_0\in L^{\frac{2}{\alpha-1}}(\RR^2)$. There is a constant $C>0$ such that for all $t>0$ and $x \in \RR^2$, we have 
\begin{align}\label{eq:uppercrit}
\theta(t,x)\le C P_t|\theta_0|(x).
\end{align}
\end{lemma}
\begin{proof}
Let  $v=R^{\perp}\theta$ and consider the linear equation
\begin{align}\label{eq:MM}
\partial_t u = \Delta^{\alpha/2} u + v \cdot \nabla u.
\end{align}
By \cite[Corollary 1.4]{2013-YM-HM-jfa}, the fundamental solution $\tilde{p}(t,x,y)$ of \eqref{eq:MM} is bounded by $p(t,x-y)$, that is
\begin{align}\label{eq:MM1}
	\tilde{p}(t,x,y) \le c\,p(t,x-y), \qquad t>0, \; x,y \in \RR^2.
\end{align}
Indeed, taking $\lambda=\frac{2(2-\alpha)}\alpha$ and $q=\infty$ in \cite[Corollary 1.4]{2013-YM-HM-jfa}, we only need to show that all required assumptions are satisfied, i.e. $\nabla v=0$ and 
\begin{align}
	v &\in L^{2,\frac{2}{\alpha} - \frac{\lambda}{2}}((0,\infty); \mathcal{L}^{\frac{4}{\alpha},\lambda}(\RR^2)), \label{eq1:uppercrit}\\
v &\in L_{loc}^{\infty}((0,\infty); L_{loc}^1(\RR^2)),  \label{eq2:uppercrit} \\
v &\in L^{1,\frac{1}{\alpha}}((0,\infty); L^1_{u\ loc}(\RR^2)), \label{eq3:uppercrit}
\end{align}
where $\mathcal{L}^{p,\lambda}(\RR^2)$ is a Campanato space. Since $\lm = \frac{2(2-\alpha)}{\alpha}<2$ for $\alpha>1$, the Campanato space $\mathcal{L}^{\frac4\alpha,\lambda}(\RR^2)$ reduces to the Morrey space $L^{\frac4\alpha,\lambda}(\RR^2)$, see, e.g. \cite{RSS}.
 Clearly, we have $\nabla v=0$.
Furthermore, by (\ref{eq:Rieszineq}), (\ref{eq:thetaLpnorm}) and H\"older inequality,
\begin{align*}
\|R^{\perp}\theta(u,\cdot)\|_{L^{\frac{4}{\alpha},\frac{2(2-\alpha)}{\alpha}}}
&\le \sup_{x\in\RR^2, r>0}\(r^{-\frac{2(2-\alpha)}{\alpha}}\int_{B(x,r)}|R\theta(t,z)|^{\frac4{\alpha}}dz\)^{\frac{\alpha}4}\\
&\le \sup_{x\in\RR^2, r>0}\(r^{-\frac{2(2-\alpha)}{\alpha}}\(\int_{B(x,r)}dz\)^{\frac{2-\alpha}\alpha}\(\int_{\RR^2}|R\theta(t,z)|^{\frac2{\alpha-1}}dz\)^{\frac{2(\alpha-1)}{\alpha}}\)^{\frac{\alpha}4}\\
&\le c \|\theta_0\|_{\frac{2}{\alpha-1}}.
\end{align*}
Hence,
\begin{align*}
\norm{v}_{L^{2,\frac{2}{\alpha} - \frac{\lambda}{2}}\big((0,\infty);L^{\frac{4}{\alpha},\frac{2(2-\alpha)}{\alpha}}\big)} 
= \sup_{t>0}\sup_{0<s<t} \((t-s)^{-1}\int_s^t \norm{R\theta(u,\cdot)}_{L^{\frac{4}{\alpha},\frac{2(2-\alpha)}{\alpha}}}^{2}du\)^{\frac12}
\le c \norm{\theta_0}_{\frac{2}{\alpha-1}},
\end{align*}
which gives \eqref{eq1:uppercrit}. Next, \eqref{eq2:uppercrit} is an immediate consequence of \eqref{eq:RthetaLinfty}. Finally, by \eqref{eq:RthetaLinfty}, we have
$$\|R^{\perp}\theta(t,\cdot)\|_{L^1_{u\ loc}(\RR^2)}:=\sup_{x\in\RR^2}\int_{B(x,1)}|R^{\perp}\theta(t,y)|dy\lesssim \|R^{\perp}\theta\|_\infty\lesssim t^{-(\alpha-1)/\alpha}. $$
Consequently,
\begin{align*}
\sup_{t>0}\sup_{0<s<t}\((t-s)^{-1/\alpha}\int_s^t\|R\theta(u,\cdot)\|_{L^1_{u\ loc}(\RR^2)}du\)\lesssim \sup_{t>0}\sup_{0<s<t}\((t-s)^{-1/\alpha}\int_0^{t-s}u^{-(\alpha-1)/\alpha}du\)\le\alpha ,
\end{align*}
which yields \eqref{eq3:uppercrit}. 
Now consider \eqref{eq:MM} with initial condition $u_0 = \theta_0$. Clearly, $$\theta(t,x) = \int_{\RR^2}\tilde{p}(t,x,y) \theta_0(y) dy$$ is a solution to this problem and  \eqref{eq:MM1} gives us
\begin{align*}
|\theta(t,x)| \le \int_{\RR^2} \tilde{p}(t,x,y) |\theta_0(y)| dy \le c \int_{\RR^2} p(t,x,y) |\theta_0(y)| dy = c\,P_t |\theta_0|(x).
\end{align*}

\end{proof}

\begin{proposition}\label{prop:Rtheta-lim} Assume $\theta_0\in L^{\frac{2}{\alpha-1}}(\RR^2)$. We have
\begin{equation}\label{eq:limRtheta}\lim_{t\rightarrow0}\|t^{\frac{\alpha-1}\alpha}R\theta(t,\cdot)\|_\infty=\lim_{t\rightarrow\infty}\|t^{\frac{\alpha-1}\alpha}R\theta(t,\cdot)\|_\infty=\lim_{|x|\rightarrow\infty}\sup_{t>0}\left|t^{\frac{\alpha-1}\alpha}R\theta(t,x)\right|=0.
\end{equation}
\end{proposition}

\begin{proof} We will use the equality \eqref{eq:Duhamel}.
The required results for the term $R_iP_t\theta_0(x)$ have been provided in Proposition \ref{prop:RPthetaest}, so what has left is to deal with the integral term.
  By (\ref{eq:uppercrit}) and (\ref{eq:limPtf}),  for every $\delta>0$ there are  $t_\delta, T_\delta>0$ such that $\|\theta(s,\cdot)\|_{\infty}<\delta s^{-(\alpha-1)/\alpha}$ for $s<t_\delta$ or $s > T_\delta$. We fix some $p>\frac{2}{\alpha-1}$.
Consequently, by \eqref{Eq3:tmp2},
\begin{align}\label{eq1:limRtheta}
\left|\int_0^t\int_{\R^2}R_i\nabla p(t-s,x,y)\cdot \Rp\theta_s(y)\theta(s,y)dy\,ds\right| \le \delta c t^{-\frac{\alpha-1}{\alpha}}, \quad x \in \RR^d,\; t\le t_\delta,
\end{align}
which gives the first limit in \eqref{eq:limRtheta}.
Now, let $t > 2T_\delta$. 
By \eqref{Eq3:tmp2},  we get 
\begin{align*}
\left|\int_{T_\delta}^t\int_{\R^2}R_i\nabla p(t-s,x,y)\cdot \Rp\theta_s(y)\theta(s,y)dy\,ds\right| \le \delta c t^{-\frac{\alpha-1}{\alpha}}, \quad x \in \RR^d,\; t>{T_\delta}.
\end{align*}
Next, by \eqref{Eq1:tmp1} (with $f= R_i\nabla p$ and $g= \theta\Rp\theta$) and \eqref{eq:thetaLpnorm},
\begin{align*}
\left|\int_{0}^{T_\delta}\int_{\R^2}R_i\nabla p(t-s,x,y)\cdot \Rp\theta(s,y)\theta(s,y)dy\,ds\right|
&\lesssim \int_{0}^{T_\delta} (t-s)^{-\frac1\alpha-\frac2{\alpha p}}s^{-\frac{2(\alpha-1)}{\alpha} +\frac2{\alpha p}} \,ds\\
& \lesssim t^{-\frac1\alpha-\frac2{\alpha p}}\int_0^{T_\delta}	 s^{-1+\frac{p(2-\alpha)+2}{\alpha p}}ds\\
& = c t^{-\frac{\alpha-1}{\alpha}- \frac{p(2-\alpha)+2}{\alpha p}}.
\end{align*}
This proves  the second limit in \eqref{eq:limRtheta}. Finally, we deal with $\lim_{|x|\rightarrow\infty}\sup_{t>0}\left|t^{\frac{\alpha-1}\alpha}R\theta(t,x)\right|=0$. 
By  (\ref{eq:uppercrit}) and (\ref{eq:limPtf3}), for every $\varepsilon\in (0,1)$ there exists $r_{\varepsilon}$  such that $\sup_{s>0}|s^{\frac{\alpha-1}\alpha}\theta(s,y)|<\varepsilon$ for $|y|>r_{\varepsilon}$. Then, by \eqref{Eq3:tmp2},
\begin{align*}
\left|\int_0^t\int_{B(0,r_\varepsilon)^c} R_i\nabla p(t-s,x,y)\cdot R^\perp\theta(s,y)\theta(s,y)dy\,ds\right| \le \varepsilon c t^{-\frac{\alpha-1}{\alpha}}.
\end{align*}
Furthermore,  by (\ref{eq:estRp}), 
$$|R_i\nabla p(t-s,x,y)| \le c (t-s)^{-\frac{1}{\alpha}} |x-y|^{-2} < \varepsilon r_\varepsilon^{-2} (t-s)^{-\frac{1}{\alpha}}$$
 for $y \in B(0,r_\eps)$ and $|x|$ sufficiently large. Hence, by \eqref{eq:thetaLpnorm}  and (\ref{eq:RthetaLinfty}), we get 
\begin{align*}
\left|\int_0^{t}\int_{B(0,r_\varepsilon)}R_i\nabla p(t-s,x,y)\cdot R^\perp\theta(s,y)\theta(s,y)dy\,ds\right| 
&\le  \varepsilon r_\varepsilon^{-2} \int_0^{t} \int_{B(0,r_\varepsilon)}(t-s)^{-\frac{1}\alpha} s^{-\frac{2\alpha-2}{\alpha}} dy\,ds\\
& \le   c\,  \varepsilon t^{-\frac{\alpha-1}{\alpha}},
\end{align*}
which  ends the proof.

%

\end{proof}
\begin{proof}[\bf Proof of Theorem \ref{thm:limt}.] First, observe that by (\ref{eq:uppercrit}) and semi-group property of $p(t,x)$, 
\begin{align}
&\left|\int_0^t\int_{\R^2}\nabla p(t-s,x-y)\cdot R\theta(s,y)\theta(s,y)dy\,ds\right| \notag\\
&\lesssim \int_0^t(t-s)^{-\frac1\alpha}\norm{R\theta(s,\cdot)}_\infty \int_{\R^2} p(t-s,x-y)P_s\theta_0(y)dy\,ds \notag\\
&= P_t\theta_0(x) \int_0^t(t-s)^{-\frac1\alpha}\norm{R\theta(s,\cdot)}_\infty ds. \label{eq1:limt}
\end{align}
By Proposition \ref{prop:Rtheta-lim}, for every $\varepsilon>0$ there are $t_\varepsilon>0$ and $T_\varepsilon$ such that $\|R\theta(t,\cdot)\|_\infty\le \varepsilon t^{-(\alpha-1)/\alpha}$ for $t<t_\varepsilon$ or $t>T_\varepsilon$. Hence, by \eqref{eq1:limt}, for $t<t_\varepsilon$, we have 
\begin{align*}
\left|\int_0^t\int_{\R^2}\nabla p(t-s,x-y)\cdot R\theta(s,y)\theta(s,y)dy\,ds\right| &\le \varepsilon P_t\theta_0(x)\int_0^t (t-s)^{-\frac1\alpha} s^{-\frac{\alpha-1}{\alpha}}ds \\
&=\varepsilon P_t\theta_0(x)\mathcal{B}\(\tfrac{\alpha-1}{\alpha}, \tfrac1\alpha\).
\end{align*}
Thus, \eqref{eq:Duhamel} gives us 
$$\left|\frac{\theta(t,x)}{P_t\theta_0(x)}-1\right|\lesssim \varepsilon,$$
which proves the first limit in \eqref{eq:limt}. Similarly, we get for $t>2T_\varepsilon$,
\begin{align*}
&\left|\int_0^t\int_{\R^2}\nabla p(t-s,x-y)\cdot R\theta_s(y)\theta(s,y)dy\,ds\right|\\
&\le  P_t\theta_0(x) \(c\int_0^{T_\varepsilon} (t-s)^{-\frac1\alpha} s^{-\frac{\alpha-1}\alpha} ds + \varepsilon\int_{T_\varepsilon}^t (t-s)^{-\frac1\alpha}  s^{-\frac{\alpha-1}\alpha} ds\)\\
&\lesssim   P_t\theta_0(x)\(ct^{-\frac1\alpha}T_\varepsilon^{\frac{1}{\alpha}} + \varepsilon\mathcal{B}\(\tfrac{\alpha-1}{\alpha}, \tfrac1\alpha\)\),
\end{align*}
which is less than $2\varepsilon \mathcal{B}\(\tfrac{\alpha-1}{\alpha}, \tfrac1\alpha\) P_t\theta_0(x)$ for $t$ large enough. Hence, we obtain the second limit in \eqref{eq:limt}.

Finally,  the previous limit lets us prove  $\lim_{|x|\rightarrow\infty}\sup_{t>0}\left|\frac{\theta(t,x)}{P_t\theta_0(x)}-1\right|=0$  by showing that
$$\lim_{|x|\rightarrow\infty}\sup_{0<t<T}\left|\frac{\theta(t,x)}{P_t\theta_0(x)}-1\right|=0$$
holds for any $T>0$. By (\ref{eq:limRtheta}), for every $\varepsilon>0$ there is $M>0$ such that $|t^{\frac{\alpha-1}\alpha}R^\perp \theta(t,x)|<\varepsilon$ for $|x|>M$. Hence, by (\ref{eq:estgradp}) and (\ref{eq:uppercrit}), we get 
\begin{align*}
&\left|\int_0^t\int_{|y|>M}\nabla p(t-s,x-y)\cdot R\theta(s,y)\theta(s,y)dy\,ds\right| \\
&\lesssim \varepsilon\int_0^t(t-s)^{-\frac1\alpha}s^{-\frac{\alpha-1}\alpha}\int_{\R^2}p(t-s,x-y)P_s\theta_0(y)dy\,ds = \varepsilon\mathcal{B}\(\tfrac{\alpha-1}{\alpha}, \tfrac1\alpha\) P_t\theta_0(x).
\end{align*}
Next, for $|x|>2M$ and $t<T$,  by (\ref{eq:estgradp}),
\begin{align*}
&\left|\int_0^t\int_{|y| \le M}\nabla p(t-s,x-y)\cdot R\theta(s,y)\theta(s,y)dy\,ds\right|\\
&\lesssim \left|\int_0^ts^{-(\alpha-1)/\alpha}\int_{|y|\le M}\frac{1}{|x|^{\frac{1}{\alpha}}} p(t-s,x-y)P_s\theta_0(y)dy\,ds\right| 
\le \frac{\alpha T^{\frac{1}{\alpha}}}{|x|^{\frac{1}{\alpha}}} P_t\theta_0(x).
\end{align*}
This ends the proof.
\end{proof}

\begin{proof}[\bf Proof of Theorem \ref{thm:maincrit}]
The upper bound follows from Lemma \ref{lem:uppercrit}. To prove the lower one, note that  Lemma \ref{thm:limt} implies  $\theta(t,x)\gtrsim P_t\theta_0(x)$ whenever $t\in(0,t_0)\cup (T,\infty)$ or $|x|>R$ for some $t_0, T, R>0$. Since both, $\theta(t,x)$ and $ P_t\theta_0(x)$  are continuous, they are comparable on $\overline{[t_0,T]\times B(0,R)}$ as well.
\end{proof}
In the last part of this section,  we consider the case $\theta_0 \in L^p$ with $p > \frac{2}{\alpha-1}$. As a result, we  obtain  the local in time analog of Theorem \ref{thm:maincrit}.
By Remark 3.3 in \cite{CF}, for $p>\frac{2}{\alpha-1}$, we have
$$\|\theta(t,\cdot)\|_q \lesssim t^{-\frac{2}{\alpha}\(\frac1p-\frac1q\)}, \qquad  p \le q \le \infty.$$

\begin{proposition}
For nonnegative $\theta_0\in L^{p}(\RR^2)$, $p>\frac{2}{1-\alpha}$ and $T>0$ there are constants $C_1$ and  $C_2$ (depending on $T$) such that
\begin{align*}
C_1   P_t \theta_0(x) \le \theta(t,x) \le  C_2   P_t \theta_0(x), \qquad x \in \RR^2, 0<t \le T.
\end{align*}
\end{proposition}
\begin{proof}
Let $T>0$. Let us consider the equation
$$
\begin{cases}
\partial_t u &= \Delta^{\alpha/2} u + b(t,x) \cdot \nabla u\\
u(0,x) &= \theta_0(x), 
\end{cases}
$$
where $b(t,x) = (\Rp \theta)(t,x)$.  Of  course $u(t,x) = \theta(t,x)$ is a solution to the above equation. By (\ref{eq:Rieszineq}), we have
\begin{align*}
\norm{b(t,\cdot)}_p  \le c \norm{\theta(t,\cdot)}_p \le c \norm{\theta_0}_p. 
\end{align*}
By H\"older inequality,
\begin{align*}
&\int_s^t \int_{\RR^2} \frac{p(u-s,z-x)}{(u-s)^{1/\alpha}} | b(u,z)| dzdu \le \int_s^t \frac{1}{(u-s)^{1/\alpha}} \norm{p(u-s,\cdot)}_\frac{p}{p-1}\, \norm{b(u,\cdot)}_p \,du\\
&\le c \int_s^t \frac{1}{(u-s)^{1/\alpha}} (u-s)^{\frac{2}{\alpha}(\frac{p-1}{p}-1)} du = c \int_s^t (u-s)^{-\frac{2}{\alpha p} - \frac{1}{\alpha}} du = c_1(t-s)^{1-\frac{2+p}{\alpha p}}.
\end{align*}
In the same way, we get
\begin{align*}
&\int_s^t \int_{\RR^2} \frac{p(t-u,z-x)}{(t-u)^{1/\alpha}} | b(u,z)| dzdu \le  c_1(t-s)^{1-\frac{2+p}{\alpha p}}.
\end{align*}
Note that $\frac{2+p}{\alpha p}< 1$, and consequently  $c(t-s)^{1-\frac{2+q}{\alpha q}} \le \eta + \beta (t-s)$ for arbitrary small $\eta$  and some $\beta >0$. Hence, $b$ belongs to the class $\mathcal{K}(\eta,Q)$ (see \cite[Definition 1]{2010-TJ-KS-jee}. Then, by \cite[Theorems 2 and 3]{2010-TJ-KS-jee}, the fundamental solution of the equation $\partial_t u = \Delta^{\alpha/2} u + b(t,x) \cdot \nabla u$ is locally in time comparable with $p$ and we get the assertion of the proposition.

\end{proof}

\section{Gradient estimates}

In this section we derive the pointwise estimates for $\nabla^\k \theta$. 
Recall that for a multi-index $\k = (k_1,k_2)\in \N^2$, we put $|\k| = k_1+k_2$. Note that
\begin{align}\label{eq:summulti}
	\nabla^\k (fg) = \sum_{\mathbf{m} + \mathbf{n} = \k} c_{\mathbf{m},\mathbf{n}}\nabla^{\mathbf{m}}f \;\nabla^\mathbf{n}g,
\end{align}
where the sum is taken over all multi-indices $\mathbf{m}$ and $\mathbf{n}$ such that $\mathbf{m}+\mathbf{n}=\k$.

\begin{lemma}For $\theta_0\in L^{\frac2{\alpha-1}}$, we have
\begin{align}\label{eq:aux1}
\norm{\nabla^\k R_i\theta(t,\cdot)}_\infty\lesssim t^{-\frac{|\k|+\alpha-1}\alpha}, \qquad i=1,2.
\end{align}
\end{lemma}
\begin{proof}
Let us rewrite (\ref{eq:Duhamel}) as follows,
\begin{align}\label{eq:Duhamel2}
\theta(t,\cdot)(x)&=\int_{\RR^2}p(t,x-y)\theta_0(y)dy+\int_0^{t/2}\int_{\R^2}\nabla p(t-s,x-y)\cdot R^\perp\theta(s,y)\theta(s,y)dy\,ds\\\nonumber
&+\int_{t/2}^t\int_{\R^2} \nabla p(t-s,y)\cdot R^\perp\theta(s,x-y)\theta(s,x-y)dy\,ds.
\end{align}
Since the Riesz transform commutes with derivatives, by \eqref{eq:Duhamel2} and  \eqref{eq:summulti}, we get 
\begin{align}\label{eq:nablaDuhamel}
&\nabla^\k R_i\theta(t,x)\nonumber\\
&=\int_{\RR^2}R_i\nabla^\k p(t,x-y)\theta_0(y)dy+\int_0^{t/2}\int_{\R^2}R_i\(\nabla^\k \nabla p(t-s,x-y)\)\cdot R^\perp\theta(s,y)\theta(s,y)dy\,ds\\\nonumber
&\ \ \ +\sum_{\ki+\kii=\k}c_{\ki,\kii}\int_{t/2}^t\int_{\R^2} R_i\(\nabla p(t-s,y)\)\cdot R^\perp\(\nabla^{\ki}\theta(s,x-y)\)\nabla^{\kii}\theta(s,x-y)dy\,ds,
\end{align}
where $\ki,\kii\in\N^2$. Hence, by H\"older inequality, \eqref{eq:thetaLpnorm}, \eqref{eq:lpestp}, \eqref{eq:Rieszineq}  and \eqref{eq:thetaLp}, for $p>\frac{2}{\alpha-1}$, we get 
\begin{align*}
\|\nabla^\k R_i\theta(t,\cdot)\|_\infty&\lesssim  \|\nabla^\k p(t,\cdot)\|_{\frac{2}{3-\alpha}} \, \|\theta_0\|_{\frac{2}{\alpha-1}}\\
&\ \ \ +\int_0^{t/2}s^{-\frac{\alpha-1}\alpha} \|\nabla^\k\nabla p(t-s,\cdot)\|_{\frac{2}{3-\alpha}} \, \|\theta(s,\cdot)\|_{\frac{2}{\alpha-1}}\,ds\\
&\ \ \ +\sum_{\ki+\kii=\k} c_{\ki,\kii} \int_{t/2}^t\| \nabla p(t-s,\cdot)\|_\frac{p}{p-1}\|(R^\perp\nabla^{\ki}\theta(s,\cdot)\|_{p}\,\norm{\nabla^{\kii}\theta(s,\cdot)}_\infty\,ds\\
&\lesssim  t^{-\frac{|\k|+\alpha-1}\alpha}+\int_0^{t/2}s^{-\frac{\alpha-1}\alpha} \, t^{-\frac{|\k|+1+\alpha}\alpha} \,ds\\
&\ \ \ +\sum_{\ki+\kii=\k} c_{\ki,\kii} \int_{t/2}^t (t-s)^{-\frac1\alpha-\frac{2}{\alpha p}} \, t^{-\frac{|\ki|}\alpha-\frac{\alpha-1}\alpha + \frac2{\alpha p}} \, t^{-\frac{|\kii|+\alpha-1}\alpha} \,ds\\[8pt]
&\lesssim  t^{-\frac{|\k|+\alpha-1}\alpha},
\end{align*}
as required.
\end{proof}
Next, we present a series of  auxiliary lemmas that are used in the proof of Theorem \ref{thm:gradest}.
\begin{lemma}\label{lem:PtthetaLow}
Let $0 < t_1 < t_2<\infty$ and $\theta_0 \in L^{\frac2{\alpha-1}}$. If $\norm{\theta_0}_{\frac2{\alpha-1}}>0$, then, there exists a constant $C= C(t_1,t_2,\theta_0)$ such that
\begin{align*}
	P_t |\theta_0(x)| \ge \frac{C}{(1+|x|)^{2+\alpha}}, \qquad t_2>t>t_1,\; x \in \RR^2.
\end{align*}
\end{lemma}
\begin{proof}
Since $\theta_0 \in L^{\frac{2}{\alpha-1}}$, then $\theta_0 \in L^1_{loc}$. Hence, there is $R>0$ such that $C<\int_{B(0,R)} |\theta_0(y)| dy <\infty$  for some $c>0$. By \eqref{eq:estp}, we get
\begin{align*}
P_t |\theta_0|(x) \ge \int_{B(0,R)} p(t,x-y) |\theta_0(y)|dy &\ge  c_1 \frac{t}{(t^{1/\alpha}+ 2R+ |x|)^{2+\alpha}} \int_{B(0,R)} |\theta_0(y)|dy \\
&\ge c c_1 \frac{t_1}{(2t_2^{1/\alpha}+ 2R+|x|)^{2+\alpha}} \ge \frac{C}{(1+|x|)^{2+\alpha}}.
\end{align*} 
\end{proof}

\begin{lemma}\label{lem:bbPtheta}
Let $\theta_0 \in L^{\frac{2}{\alpha-1}}$. Let $0<t_1<t_2 < \infty$ . There exists a constant $C$ depending on $t_1,t_2,R$ and $\theta_0$  such that for  $x \in \RR^2$, we have
\begin{align*}
\int_{D_t}\int_{B(0,R)} (t-s)^{-1/\alpha}  p(t-s,x-y)s^{-(\alpha-1)/\alpha}|\nabla^{\k}\theta(s,y)|dy\,ds \le C t^{-|\k|/\alpha} P_t |\theta_0|(x),
\end{align*}
where $D_t=(t_1,t_2) \cap (t/2,t)$.
\end{lemma}
\begin{proof}
Note that $D_t = \emptyset$ for $t \notin (t_1,2t_2)$, hence, it suffices to consider only $t_1<t<2t_2$.
By \eqref{eq:thetaLp}, 
\begin{align*}
&\int_{D_t}\int_{B(0,R)} (t-s)^{-\frac{1}{\alpha}}  p(t-s,x-y)s^{-(\alpha-1)/\alpha}|\nabla^{\k}\theta(s,y)|dy\,ds \\
& \le  c\int_{D_t}\int_{B(0,R)} (t-s)^{-\frac{1}{\alpha}}  p(t-s,x-y)s^{-\frac{\alpha-1}\alpha-\frac{\alpha-1}\alpha- \frac{|\k|}\alpha }dy\,ds \\
&\le  ct_1^{-\frac{2(\alpha-1)}{\alpha}} \big(\tfrac{t}{2}\big)^{-\frac{|\k|}\alpha } \int_{D_t}  (t-s)^{-\frac{1}{\alpha}}  P_{t-s} \mathbbm{1}_{B(0,R)}(x)\,ds =: f(t,x).
\end{align*}
Note that $p(s,y) \ge \frac{1}{c_1} >0$ for $(s,y) \in (t_1,t_2) \times B(0,R)$. Hence,
$$P_{t-s} \mathbbm{1}_{B(0,R)}(x)\le c_1\int_{\RR^2} p(t-s,x-y) p(s,y) dy = c_1 p(t,x) \le \frac{c_2}{(1+|x|)^{2+\alpha}},$$
Consequently, by Lemma \ref{lem:PtthetaLow},
\begin{align*}
f(t,x) \le c_3 t^{-\frac{|\k|}{\alpha}} \int_{t_1}^{t}  (t-s)^{-\frac{1}{\alpha}}  \frac{1}{(1+|x|)^{2+\alpha}}\,ds \le c_4t^{-\frac{|\k|}{\alpha}} P_t |\theta_0|(x).
	\end{align*}
\end{proof}

\begin{lemma}\label{lem:tech}Let $\beta>0$ be fixed. For any $v\in(0,1)$, we have
$$\int_v^1 r^{-\beta}(1-r^\alpha)^{-1/\alpha}(r^\alpha-v^\alpha)^{-1/\alpha}    dr\approx v^{-\beta}(1-v)^{1-2/\alpha}$$
with comparability constants depending only on $\alpha$ and  $\beta$.
\end{lemma}
\begin{proof} Denote the above integral by $I(v)$. Since $a^\gamma-b^\gamma\approx (a-b)a^{\gamma-1}$ for $a>b>0$  and $\gamma>0$ (see e.g Lemma 4 in \cite{MS}), we have $1-r^\alpha\approx 1-r$ and $r^\alpha-v^\alpha\approx(r-v)r^{\alpha-1}$. Hence,
$$I(v) \approx \int_v^1 r^{1/\alpha-1-\beta}(1-r)^{-1/\alpha}(r-v)^{-1/\alpha}dr. $$
For $v\ge1/4$, we estimate $ r^{1/\alpha-1-\beta} \approx 1$ and substitute $r=1-u(1-v)$, which gives us
$$I(v) \approx (1-v)^{1-2/\alpha} \int_0^1 u^{-1/\alpha}(1-u)^{-1/\alpha}du= c(1-v)^{1-2/\alpha}.$$
In the case $v<1/4$, we split the integral into $\int_v^{1/2}+\int_{1/2}^1$ and obtain
\begin{align*}
I(v) &\approx \int_v^{1/2} r^{1/\alpha-1-\beta}(r-v)^{-1/\alpha}dr+ \int_{1/2}^1 (1-r)^{-1/\alpha}dr\\
&= v^{-\beta}\int_1^{1/(2v)} u^{1/\alpha-1-\beta}(u-1)^{-1/\alpha}du+\frac{\alpha2^{(\alpha-1)/\alpha}}{\alpha-1}\\ 
&\approx v^{-\beta}+1\approx v^{-\beta}, 
\end{align*}
which is equivalent to the required formula under current assumptions.
\end{proof}

Since $\alpha>1$, we immediately obtain the following 
\begin{cor}
\label{cor:estint} Let $\beta>0$ be fixed. There is a constant $C_\beta$ such that for $v\in(0,1)$, we have
$$\int_v^1 r^{-\beta}(1-r^\alpha)^{-1/\alpha}(r^\alpha-v^\alpha)^{-1/\alpha}    dr\le C_\beta v^{-\beta}(1-v)^{-1/\alpha}.$$
\end{cor}

\begin{lemma}\label{lem:s-ie} 
Fix $\gamma \in (0,\frac1\alpha)$. For any measurable function $f \colon \RR \times \RR^2 \to \RR$, define the operator
\begin{align}\label{eq:T_g}
T_\gamma f(t,x)=t^{\gamma}\int_{0}^ts^{-\gamma-\frac{\alpha-1}{\alpha}}(t-s)^{-\frac{1}{\alpha}} P_{t-s}|f|(s,x)\,ds.
\end{align}
Suppose $T_\gamma f(t,x) < \infty$ and $f$ satisfies the inequality
\begin{align}\label{aux4}
f(t,x) \le C P_t\theta_0(x) + \eta T_\gamma f(t,x), \qquad t>0, \; x \in \RR^2,
\end{align}
for some constants $C,\eta >0$. If $\eta$ is sufficiently small, then there exists a constant $M>0$ such that
$$f(t,x)\le M P_t|\theta_0|(x), \qquad t>0, \; x \in \RR^2.$$
\end{lemma}
\begin{proof}
 Applying estimate \eqref{aux4} of $f$ to \eqref{eq:T_g}, we get
\begin{align}\notag
T_\gamma f(t,x)&\le  t^{\gamma}\int_{0}^ts^{-\gamma-(\alpha-1)/\alpha}(t-s)^{-1/\alpha}\int_{\R^2}  p(t-s,x-y)\\\notag
&\ \ \ \ \times\(C P_s|\theta_0|(y)+\eta \int_{0}^su^{-(\alpha-1)/\alpha}(s-u)^{-1/\alpha} P_{s-u}|f|(u,y) du\)dy\,ds\\\notag
&=C\mathcal{B}\(1-\gamma-\tfrac{\alpha-1}\alpha, 1-\tfrac{1}{\alpha}\) P_t|\theta_0|(x)\\\notag
&\ \ \ +t^{\gamma}\eta \int_{0}^t\int_{0}^ss^{-\gamma}(su)^{-(\alpha-1)/\alpha}[(t-s)(s-u)]^{-1/\alpha}  P_{t-u} |f|(u,x)\,duds\\\notag
&=C\mathcal{B}\(1-\gamma-\tfrac{\alpha-1}\alpha, 1-\tfrac{1}{\alpha}\)  P_t|\theta_0(x)|\\\label{aux5}
&\ \ \ +\eta t^{\gamma}\int_{0}^tu^{-(\alpha-1)/\alpha}P_{t-u} |f|(u,x)\int_{u}^ts^{-\gamma-(\alpha-1)/\alpha}[(t-s)(s-u)]^{-1/\alpha}  \,dsdu,
\end{align}
where $\mathcal{B}$ is the beta function. Using  Corollary \ref{cor:estint} with $\beta=\gamma\alpha$ and $v=(u/t)^{1/\alpha}$,
we estimate the last inner integral in  \eqref{aux5}  as follows
\begin{align*}
&\int_{u}^ts^{-\gamma-(\alpha-1)/\alpha}[(t-s)(s-u)]^{-1/\alpha}  \,ds\\
&=t^{-\gamma-1/\alpha}\int_{u/t}^1s^{-\gamma-(\alpha-1)/\alpha}[(1-s)(s-\frac ut)]^{-1/\alpha}  \,ds\\
&=t^{-\gamma-1/\alpha}\int_{(u/t)^{1/\alpha}}^1r^{-\gamma \alpha}[(1-r^\alpha)(r^\alpha-\frac ut)]^{-1/\alpha}  \,ds\\
&\le c_\gamma u^{-\gamma}(t-u)^{-1/\alpha}.
\end{align*}
This yields $T_\gamma f(t,x)\le  C\mathcal{B}\(1-\gamma-\tfrac{\alpha-1}\alpha, 1-\tfrac{1}{\alpha}\) P_t|\theta_0(x)|+ \eta c_\gamma T_\gamma f(t,x)$. Now, for $\eta<\frac{1}{c_\gamma}$, we get 
$$T_\gamma f(t,x)\le \frac{C\mathcal{B}\(1-\gamma-\tfrac{\alpha-1}\alpha, 1-\tfrac{1}{\alpha}\)}{1-\eta c_\gamma}P_t|\theta_0|(x),$$
which ends the proof.

\end{proof}

\begin{proof}[\bf Proof of Theorem \ref{thm:gradest}.]
We will use induction with respect to $|\k|$. For $|\k|=0$ the assertion is true by Lemma \ref{lem:uppercrit}. Assume now that (\ref{eq:gradest}) holds for all multi-indices $\k'$ such that $|\k'|\le |\k|-1$ for some multi-index $\k$, $|\k|\geq1$.  We use \eqref{eq:Duhamel2} and, analogously as in (\ref{eq:nablaDuhamel}), we obtain
\begin{align*}
\nabla^\k\theta(t,x)&=\nabla^\k P_t\theta_0(x)+\int_0^{t/2}\int_{\R^2}\(\nabla^\k \nabla p(t-s,x-y)\)\cdot R^\perp\theta(s,y)\theta(s,y)dy\,ds\\\nonumber
&\ \ \ +\sum_{\ki+\kii=\k}c_{\ki,\kii}\int_{t/2}^t\int_{\R^2} \(\nabla p(t-s,x-y)\)\cdot R^\perp\(\nabla^{\ki}\theta(s,y)\)\nabla^{\kii}\theta(s,y)dy\,ds.
\end{align*}
As mentioned in  Introduction,  \eqref{eq:estgradp} implies
\begin{align*}
|\nabla^\k P_t\theta_0(x)|\leq\int_{\RR^2}|\nabla^\k p(t,x-y)\theta_0(y)|dy \lesssim  t^{-\frac{|\k|}\alpha} P_t|\theta_0|(x).
\end{align*} 
Next, by \eqref{eq:estgradp}, Proposition \ref{prop:Rtheta-est}, Lemma \ref{lem:uppercrit} and semigroup property, we get
\begin{align*}
&\left|\int_0^{t/2}\int_{\R^2}\(\nabla^\k \nabla p(t-s,x-y)\)\cdot R^\perp\theta(s,y)\theta(s,y)dy\,ds \right| \\ 
&\lesssim t^{-(|\k|+1)/\alpha}\int_0^{t/2}s^{-(\alpha-1)/\alpha}\int_{\R^2} p(t-s,x-y)P_s|\theta_0|(y)dy\,ds \\
&= c t^{-\frac{|\k|}\alpha} P_t|\theta_0|(x). 
\end{align*}
Hence,  using the induction assumption for $|\kii|\le|\k|-1$ together with  \eqref{eq:estgradp}, \eqref{eq:uppercrit}, \eqref{eq:aux1} and semi-group property of $p(t,x)$, we conclude
\begin{align}\nonumber
|\nabla^\k\theta(t,x)|&\lesssim t^{-\frac{|\k|}\alpha} P_t|\theta_0|(x) +\sum_{\substack{\ki+\kii=\k \\|\kii|\le |\k|-1}} c_{\ki,\kii} t^{-(|\k|+\alpha-1)/\alpha}\int_{t/2}^t(t-s)^{-1/\alpha}\int_{\R^2}  p(t-s,x-y)P_s|\theta_0|(y)dy\,ds\\\nonumber
&\qquad\qquad\qquad\;\; +\int_{t/2}^t\int_{\R^2} \left|\nabla p(t-s,x-y)\cdot R^\perp\theta(s,y)\nabla^{\k}\theta(s,y)\right|dy\,ds \,\\\label{aux2}
&\lesssim t^{-\frac{|\k|}\alpha} P_t|\theta_0|(x)+\int_{t/2}^t(t-s)^{-\frac1\alpha}\int_{\R^2}  p(t-s,x-y)|R^\perp\theta(s,y)||\nabla^{\k}\theta(s,y)|dy\,ds.
\end{align}
Let $\varepsilon>0$, to be fixed later. By \eqref{eq:limRtheta}, there are $t_1,t_2,R>0$ such that $|s^{(\alpha-1)/\alpha}R^\perp\theta(s,y)|<\varepsilon$ for $(s,y)\notin D=(t_1,t_2)\times B(0,R) $. Thus
\begin{align*}
|\nabla^\k\theta(t,x)|
&\le c	 t^{-|k|/\alpha} P_t|\theta_0|(x)+\varepsilon\int_{t/2}^t(t-s)^{-1/\alpha}\int_{\R^2}  p(t-s,x-y)s^{-(\alpha-1)/\alpha}|\nabla^{\k}\theta(s,y)|dy\,ds\\
&\ \ \ +\int_{t_1 \vee t/2}^{t_2 \land t}\int_{B(0,R)} (t-s)^{-1/\alpha}  p(t-s,x-y)s^{-(\alpha-1)/\alpha}|\nabla^{\k}\theta(s,y)|dy\,ds.
\end{align*}
By Lemma \ref{lem:bbPtheta}, the last integral is bounded by $ t^{-|\k|/\alpha} P_t\theta_0(x)$. 
 This gives us 
\begin{align*}
|\nabla^\k\theta(t,x)|\le ct^{-|\k|/\alpha} P_t|\theta_0|(x)+\varepsilon\int_{t/2}^t(t-s)^{-1/\alpha}\int_{\R^2}  p(t-s,x-y)s^{-(\alpha-1)/\alpha}|\nabla^{\k}\theta(s,y)|dy\,ds.
\end{align*}
Now, denote $f_\k(t,x) = t^{|\k|/\alpha}|\nabla^\k\theta(t,x)|$.  Then, for any $\gamma \in (0,1/\alpha)$,
\begin{align*}
f_\k(t,x) &\le c P_t|\theta_0|(x)+ \varepsilon\int_{t/2}^t(t-s)^{-1/\alpha}\int_{\R^2}  p(t-s,x-y)s^{-(\alpha-1)/\alpha} t^{|\k|/\alpha} |\nabla^{\k}\theta(s,y)|dy\,ds \\
&\le c P_t|\theta_0|(x)+ \varepsilon 2^{|\k|/\alpha}\int_{t/2}^t(t-s)^{-1/\alpha} s^{-(\alpha-1)/\alpha}   P_{t-s} f_\k(s,x) \,ds\\
&\le c P_t|\theta_0|(x)+ \varepsilon 2^{|\k|/\alpha} T_\gamma f_\k(t,x),
\end{align*}
where $T_\gamma$ is defined in Lemma \ref{lem:s-ie}. Since $\varepsilon$ may be choosen arbitrary small, by Lemma \ref{lem:s-ie},
$$
|\nabla^\k\theta(t,x)|\le M t^{-|\k|/\alpha} P_t|\theta_0|(x).
$$
The proof is complete.
\end{proof}


\end{document}